\documentclass[10pt]{article}
\usepackage{amsmath, amssymb}
\usepackage{amsfonts}
\usepackage{amsthm}
\usepackage{amsopn}
\usepackage{amscd}
\usepackage[all]{xy}

 1
 1
 1

\newcommand{\mf}[1]{\mathfrak{#1}}

\theoremstyle{plain}
\newtheorem{thm}{Theorem}

\newtheorem*{cor}{Corollary}
\newtheorem{corr}{Corollary}

\newtheorem*{lem}{Lemma}

\theoremstyle{definition}

\theoremstyle{definition}
\newtheorem*{remark}{\textbf{Remark}}

\author{Soumya D. Sanyal}
\title{Bruns}

\begin{document}

\begin{center} 
`` \textbf{\large{``Jede'' endliche freie Aufl\"{o}sung ist freie Aufl\"{o}sung\\ eines von drei Elementen erzeugten Ideals}}'' \\
\vspace{5mm}
\textbf{Winfried Bruns}\\
\vspace{5mm}
\textbf{Journal of Algebra \textbf{39}, 429-439 (1976)}\\
\vspace{5mm}
\noindent A translation into English:\\
\vspace{5mm}
\textbf{`Every' finite free resolution is a free resolution of an ideal generated by three elements.}\\
\vspace{5mm}
\noindent Translation by: Soumya D. Sanyal \footnote{\noindent Soumya Deepta Sanyal, 28 Math. Sci. Bldg., Columbia, MO 65211; sds2p8@mail.mizzou.edu\\} $^,$ \footnote{Translator's note: The translator gratefully acknowledges support from the AMS, the NSF and the AMS Mathematics Research Communities program. The translation was begun during the 2010 AMS MRC in Commutative Algebra.\\

\noindent The translator thanks H. Srinivasan for helpful comments on the translation, A. Heinecke for proofreading the translation and C. Chiasson for assistance with the translation of Remark 4 to Theorem 3.\\

\noindent The translator thanks W. Bruns for helpful comments and for permitting the distribution of this translation.\\}\\
\vspace{5mm}
June 29, 2010\\
Present draft: December 24, 2010\\
\end{center}

\setcounter{section}{0}

\section{Introduction} In [4] and [7], Burch and Kohn proved that when $R$ is a commutative Noetherian ring, and $n$ a natural number that is the homological dimension \footnote{Translator's note: The homological dimension of a module is now more commonly known as its projective dimension.} of a torsionless $R$-module, then there exists an ideal in $R$ generated by at most 3 elements, whose homological dimension is precisely $n$. Buchsbaum and Eisenbud [3, pg. 135, Conjecture; see also Theorem 7.2] have further conjectured that `every' finite free resolution is a free resolution of such an ideal. We prove the following generalisation of their conjecture:\\

Let:\\

\begin{center}
\mbox{%

\xymatrix{
{0} \ar[r] & {F_n} \ar[r]^-{f_n} & {F_{n-1}} \ar[r] & {\dots} \ar[r] & {F_{m+1}} \ar[r]^-{f_{m+1}} & {F_m} \ar[r] & {M} \ar[r] & 0\\
}
}
\end{center}

\noindent be a projective resolution of $M$, where $F_m$ is a free $R$-module and $M$ is an m-torsionless $R$-module (or equivalently, $M$ is an m-th syzygy module) having a well defined rank. Let $r$ be the rank of the submodule Im$(f_{m+1})$. Then there exist homomorphisms:\\

$$c: F_m \rightarrow R^{r+m}; f_m: R^{r+m} \rightarrow R^{2m-1}; f_j: R^{2j+1} \rightarrow R^{2j-1}, j=1, \dots , m-1 $$

\noindent such that if $f'_{m+1}:=c \circ f_{m+1}$, the sequence:\\

\begin{center}
\mbox{%

\xymatrix{
{0} \ar[r] & {F_n} \ar[r]^-{f_n} & {F_{n-1}} \ar[r] & {\dots} \ar[r] & {F_{m+2}} \ar[r]^-{f_{m+2}} & {F_{m+1}} \ar[r]^-{f'_{m+1}} & {R^{r+m}} \\ & \ar[r]^-{f_m} & {R^{2m-1}} \ar[r] & {\dots} \ar[r] & {R^5} \ar[r]^-{f_2} & {R^3} \ar[r]^-{f_1} & {R} \\
}
}
\end{center}

\noindent is exact. (Theorem 3). When $m=2$ and the $F_k$, $k=2, \dots n$ are free modules, we obtain the conjecture of Buchsbaum and Eisenbud.\\

Our result also shows that any natural number that is the homological dimension of an $m$-torsionless $R$-module is the homological dimension of an $m$-torsionless $R$-module of rank $m$ generated by at most $2m+1$ elements.\\

The maps $f_i$ and $c$ above are constructed in Theorem 2, which states that if firstly $M$ is an $m$-torsionless $R$-module of rank $r>m$ and of finite homological dimension, and if secondly $g: R^n \rightarrow M$ is an epimorphism, then there exists a basis element $x$ of $R^n$ such that $M/{Rg(x)}$ is $m$-torsionless and of rank $r-1$. This result requires that $g(x)$ can be included as part of a basis for the free $R_{\mathfrak{p}}$-module $M_{\mathfrak{p}}$ for all prime ideals $\mathfrak{p}$ of $R$ having the property that depth$(R_{\mathfrak{p}}) \leq m$. Buchsbaum and Eisenbud have shown the existence of such elements $x \in R^n$, assuming certain hypotheses on the Krull dimension of the ring $R/\mathfrak{p}$. We are able to prove a similar statement, in which we replace their hypotheses on the Krull dimension of $R/\mathfrak{p}$ with conditions on the difference between $m$ and the depth of $R_{\mathfrak{p}}$ (Theorem 1).\\

One may replace the condition on the homological dimension of $M$ in the previous paragraph with a condition on the regularity of the rings $R_{\mathfrak{p}}$ whose depth is at most $m$ and still have the result hold. For rings satisfying this condition, we are able to generalize a theorem of Bourbaki [2, pg. 76, Theor\`{e}me 6], which states that any $m$-torsionless $R$-module of rank at least $m$ is an extension of an $m$-torsionless $R$-module of rank $m$ by a free $R$-module (Corollary 2 to Theorem 2).\\

In our paper, $R$ will throughout denote a commutative Noetherian ring and all $R$-modules $M, N, \dots$ shall be finitely generated. We denote by $Q(R)$ the total quotient ring of $R$. If $M \otimes Q(R)$ is a free $Q(R)$-module, then the rank of $M$ (denoted rank $M$) is defined to be the number of elements in a basis of $M \otimes Q(R)$. We remind the reader that if any two modules in a short exact sequence have well defined ranks, then so does the third ([9, chap 6]). We denote the homological dimension of $M$ over $R$ by dh($M$), and the homological codimension (depth) by codh($M$). When $R$ is a local ring we define the (well-defined) number of elements in a minimal generating set of $M$ by $\mu(M)$. Ass $M$ denotes the (finite) set of associated primes of $M$. We denote the grade of an ideal $\mathfrak{a}$ (with respect to elements in $R$) by grad($\mathfrak{a}$); thus grad($\mathfrak{a}) = $min$ \{$codh$ (R_{\mathfrak{p}}) \mid \mathfrak{p} \supset \mathfrak{a}\}$. For notational convenience, we shall define $\mathfrak{C}_n$ to be the set of prime ideals $\mathfrak{p}$ of $R$ such that codh$(R_{\mathfrak{p}}) \leq n$.\\

We briefly explain the notion of `$m$-torsionless' for a module. Suppose $F_1 \rightarrow F_0 \rightarrow M \rightarrow 0$ is a presentation of a module $M$, and let $D(M)$ be the cokernel of the dual of the homomorphism $F_1 \rightarrow F_0$ \footnote{Translator's note: this is the `transpose' of $M$ in the sense of Auslander, as remarked in ``Linear Free Resolutions and Minimal Multiplicity'', by D. Eisenbud and S. Goto; Journal of Algebra \textbf{88}, 89-133 (1984).}. The Ext$^i_R(D(M),R)$ modules for $i \geq 1$ do not depend on the choice of presentation of $M$. Thus we may say that $M$ is $m$-torsionless provided that Ext$^i_R(D(M),R)=0$ for $i=1, \dots m$ [1]. The nomenclature derives from the fact that Ext$^1_R(D(M),R)$ and Ext$^2_R(D(M),R)$ are the kernel and the cokernel respectively of the natural homomorphism of $M$ into its bidual. Hence $M$ is $1$-torsionless (resp. $2$-torsionless) if $M$ is torsionless (resp. reflexive) in the usual sense. If dh$(M) < \infty$ there are the following equivalent [1, Theorem 4.25] characterisations of $m$-torsionlessness:\\

\begin{enumerate}
\item $(a_m)$ Every R-regular sequence of at most $m$ elements is $M$-regular (thus if $m = 1$ then $M$ is torsion-free).
\item $(b_m)$ codh$(M_{\mathfrak{p}}) \geq$ min$(m,$ codh$(R_{\mathfrak{p}}))$ for every prime ideal $\mathfrak{p}$ of $R$.
\item $(s_m)$ $M$ is an $m$-th syzygy module of a projective resolution.
\item $(t_m)$ $M$ is $m$-torsionless.
\end{enumerate}
 
When the localization $R_{\mathfrak{p}}$ is a Gorenstein ring for all prime ideals $\mathfrak{p} \in \mathfrak{C}_{m-1}$, the above equivalences also hold even if dh$(M)$ is not finite [6, Theorem 4.6]. We call rings $R$ satisfying this property on its localisations $m$-Gorenstein. In this paper we will only consider $m$-Gorenstein rings; in fact, the rings $R_{\mathfrak{p}}: \mathfrak{p} \in \mathfrak{C}_{m-1}$, shall be regular local rings.\\

In formulating Theorem 1, we have used the term `$s$-basic'. A submodule $M'$ of $M$ is called $s$-basic in $M$ at $\mathfrak{p}$ if $\mu(M/M')_{\mathfrak{p}} \leq \mu(M_{\mathfrak{p}}) - s$. An element $x \in M$ is called \emph{basic} at $\mathfrak{p}$ if $Rx$ is 1-basic. In the proof of Theorem 1 we consider the set of prime ideals of $R$ such that $\mu(M_{\mathfrak{p}})>n$. This set is a variety in Spec$(R)$: if $I_n(M)$ is the sum of colon ideals $(N:M)$ where $N$ is generated by at most $n$ elements, then $\mu(M_{\mathfrak{p}})>n$ if and only if $\mathfrak{p} \supset I_n(M)$.\\

\section{On the existence of basic elements}

In this section we shall prove Theorem 1. The theorem asserts, under certain conditions, the existence of elements $y \in M$ that are basic at all $\mathfrak{p} \in \mathfrak{C}_m$. The result will allow us to pass from an $m$-torsionless $R$-module $M$ to an $m$-torsionless $R$-module $M/{Ry}$. Theorem 1 is largely analogous to [5, Theorem A]; we replace their condition on the Krull dimension of $R/{\mathfrak{p}}$ by a different condition on $\mathfrak{p}$ and an additional condition on $M$. (Remark 1 following the proof of Theorem 1 indicates a special case of this result).\\

\begin{thm} Let $n \geq 0$ be a natural number and $M$ be an $R$-module such that for all prime ideals $\mathfrak{p}, \mathfrak{q} \in \mf{C}_{n}$ we have that \textup{codh}$(R_{\mathfrak{p}}) >$ \textup{codh}$(R_{\mathfrak{q}})$ implies $\mu(M_{\mathfrak{p}}) \geq \mu(M_{\mathfrak{q}})$. Then:\\

\begin{enumerate}
\item If $\mu(M_{\mathfrak{p}})>n$ for every prime ideal $\mathfrak{p} \in \mathfrak{C}_n$, then there exists some $x \in M$ that is basic in $M$ at all $\mathfrak{p} \in \mathfrak{C}_n$.
\item Let $x_1, \dots , x_k \in M$, for $k \geq 1$ and define $M' := \Sigma_{i=1}^k{Rx_i}$. If $M'$ is \textup{min}$(k,n+1-$\textup{codh}$(R_{\mathfrak{p}}))$-basic in $M$ at all $\mathfrak{p} \in \mathfrak{C}_n$, then there exists $x' \in \Sigma_{i=2}^k{Rx_i}$ such that $x_1+x'$ is basic in $M$ at all $\mathfrak{p} \in \mathfrak{C}_n$.
\end{enumerate}
\end{thm}

We largely follow the proof of [5, Theorem A], although we use the following lemma instead of [5, Lemma 1]:\\

\begin{lem}
Let $\mathfrak{a}$ be an ideal of $R$. Then there exist only finitely many prime ideals $\mathfrak{p} \supset \mathfrak{a}$ such that \textup{codh}$(R_{\mathfrak{p}})=$ \textup{grad}$(\mathfrak{a})$.
\end{lem}

\begin{proof}
The case $\mathfrak{a}=R$ is trivial. So suppose that $\mathfrak{a} \lneq R$ and let $x_1, \dots, x_n$, where $n=$ grad$({\mathfrak{a}})$, be an $R$-regular sequence contained in $\mathfrak{a}$. Then if $\mathfrak{p} \supset \mathfrak{a}$, $x_1, \dots, x_n$ is also an $R_{\mathfrak{p}}$-regular sequence; thus when codh$(R_{\mathfrak{p}})=n$, we have:\\

\begin{center}
codh$((R/{Rx_1+ \dots + Rx_n})_{\mathfrak{p}}) = \hspace{1mm} $codh$(R_{\mathfrak{p}}/({R_{\mathfrak{p}}x_1+ \dots + R_{\mathfrak{p}}x_n})) = 0.$\\
\end{center}
\vspace{2mm}

\noindent Thus $\mathfrak{p} \in $Ass $(R/{(Rx_1+\dots + Rx_n)})$. But there are only finitely many primes in Ass $(R/{Rx_1+\dots + Rx_n})$.
\end{proof}

\noindent \textbf{Proof of Theorem 1.} (1) follows from (2) by taking $x_1, \dots x_k$ to be members of a system of generators of $M$. We will prove (2) by induction on $k$. The case $k=1$ is trivial. Suppose then that $k>1$ and let $a_1, \dots ,a_{k-1}$ be given such that $M'' = R(x_1+a_1x_k) + \Sigma_{i=2}^{k-1}{R(x_i+a_ix_k)}$ is min$(k-1,n+1-$codh$(R_{\mathfrak{p}}))$-basic at all $\mathfrak{p} \in \mathfrak{C}_n$.\\

We begin by proving that there are only finitely many primes $\mathfrak{p} \in \mathfrak{C}_n$ at which $M'$ is not min$(k,n+2-$codh$(R_{\mathfrak{p}}))$-basic, and hence only finitely many $\mathfrak{p} \in \mathfrak{C}_n$ such that\\

\begin{center} 
$\mu(M/M')_{\mathfrak{p}}>\mu(M_{\mathfrak{p}})-$min$(k,n+2-$codh$(R_{\mathfrak{p}}))$.\\
\end{center}

We note that it is enough to prove this for fixed $t:=$codh$(R_{\mathfrak{p}})$ and fixed $s:=\mu(M_{\mathfrak{p}})-$min$(l,n+2-t)$, since $s$ and $t$ will take only finitely many values. For any $\mathfrak{q} \in \mathfrak{C}_n$ with codh$(R_{\mathfrak{q}})<t$, we obtain the following bound\\
\begin{center}
$\mu((M/M')_{\mathfrak{q}}) \leq \mu(M_{\mathfrak{q}})-$min$(k,n+1-$codh$(R_{\mathfrak{q}})) \leq \mu(M_{\mathfrak{q}})-$min$(k,n+2-t) \leq s,$\\ 
\end{center}

\noindent since for every $\mathfrak{p} \in \mathfrak{C}_n$ such that codh$(R_{\mathfrak{p}})=t$, we have $\mu(M_{\mathfrak{q}})\leq \mu(M_{\mathfrak{p}})$ by hypothesis. Thus $\mathfrak{q}$ does not contain $I_s(M/M')$, whence grad$(I_s(M/M')) \geq t$. In the case that grad$(I_s(M/M'))>t$, there are no prime ideals $\mathfrak{p}$ with codh$(R_{\mathfrak{p}})=t$ that contain $I_s(M/M')$, and in the case that grad$(I_s(M/M'))=t$ there are only finitely many such ideals, as follows from the previous lemma.\\

Now let $E$ be the finite set of primes $\mathfrak{p} \in \mathfrak{C}_n$ at which $M'$ is not min$(k,n+2-$codh$(R_{\mathfrak{p}}))$-basic in $M$. By [5, Lemma 3] there exist $a_1, \dots , a_{k-1} \in R$ such that $M'' = R(x_1+a_1x_k)+\Sigma_{i=2}^{k-1}{R(x_i+a_ix_k)}$ is min$(k-1,n+1-$codh$(R_{\mathfrak{p}}))$-basic at all $\mathfrak{p} \in E$. Note that this latter condition also holds at all $\mathfrak{p} \in \mathfrak{C}_n \setminus E$.\\

\qed

\begin{remark}
\begin{enumerate}
\item The statement and proof of Theorem 1 are valid more generally (for instance, as in [5, Theorem A]), provided that $R$ is a commutative Noetherian ring, $A$ a (not necessarily commutative) $R$-algebra that is finitely generated as an $R$-module, and $M$ a finitely generated $A$-module. For given $A$-modules $N$ the definitions of $\mu(N_{\mathfrak{p}})$ and $I_s(N)$ correspond to $\mu(A_{\mathfrak{p}},N_{\mathfrak{p}})$ and $I_s(A,N)$ respectively, as in [5]. However, we may not relax the condition that $R$ be Noetherian. Furthermore, Part 2 of Theorem 1 (which is analogous to [5, Theorem A(iib)]) can be generalised in the required manner by assuming $(a,x_1)$ is basic in $A \oplus M$ at all $\mathfrak{p} \in \mathfrak{C}_n$ for every $a \in A$, thus obtaining basic elements of the form $x_1 + ax',  x' \in \Sigma_{i=2}^{k}{Ax_i}$.
\item Several conclusions follow immediately from Theorem 1, each analogous to the colloraries of [5, Theorem A]. We give the following example, corresponding to a theorem of Serre, that indicates how to formulate these conclusions along the lines of [5]: if $P$ is a projective $R$-module with a well defined rank greater than max $\{$codh$(R_{\mathfrak{p}}) \mid \mathfrak{p} \in $Spec$(R) \}$, then $P$ splits as a direct sum of rank $1$ free modules. (For a proof, see [5, Corollary 1]; the hypotheses of Theorem 1 are satisfied since $P$ has a well defined rank).\\
\end{enumerate}
\end{remark}

\section{The reduction of ranks of $m$-torsionless modules}

In this section we show that we may `replace', in an appropriate sense, an $m$-torsionless $R$-module of rank greater than $m$ by an $m$-torsionless $R$-module of rank $m$.\\

\begin{thm}
Let $M$ be an $m$-torsionless $R$-module of rank $r \geq m$, and suppose that either \textup{dh}$(M) < \infty$ or $R_{\mathfrak{p}}$ is a regular local ring for all $\mathfrak{p} \in \mathfrak{C}_m$. Suppose further that $g: R^n \rightarrow M$ is an epimorphism, that $e_1, \dots , e_n$ constitute a basis for $R^n$ and $s := n-r+m+1$. Then there exist elements $x_s, \dots , x_n \in R^n$ satisfying the following conditions:\\

\begin{enumerate}
\item $e_1, \dots , e_{s-1}, x_s, \dots , x_n$ is a basis for $R^n$.
\item $g(x_s), \dots , g(x_n)$ are linearly independent elements of $M$.
\item $M' := M/{\Sigma_{i=s}^{n}{Rg(x_i)}}$ is $m$-torsionless and of rank $m$.
\item The natural epimorphism $p: R^n \rightarrow F := R^n/{\Sigma_{i=s}^{n}{Rx_i}}$ gives an isomorphism between \textup{Ker}$(g)$ and the kernel of the induced epimorphism $g': F \rightarrow M'$. 
\end{enumerate}
\end{thm}

\begin{proof}
We proceed by induction  on $r$. If $r=m$, there is nothing to prove. If $r>m$, it is enough to show the existence of an element $x \in R^n$ satisfying the following properties:\\

\begin{enumerate}
\item $e_1, \dots , e_{n-1}, x$ is a basis for $R^n$,
\item $g(x)$ is linearly independent,
\item $M'' := M/{Rg(x)}$ is $m$-torsionless and of rank $r-1$,
\item the canonical projection $p' := R^n \rightarrow R^n/{Rx}$ gives an isomorphism between Ker$(g)$ and the kernel of the induced epimorphism $g'' := R^n/{Rx} \rightarrow M''$.
\end{enumerate}

The inequality $r>m$ implies $\mu(M_{\mathfrak{p}})>m$ for all prime ideals $\mathfrak{p}$ in $R$. Our hypotheses on $M$ and $R$ imply that dh$(M_{\mathfrak{p}})< \infty$ for all $\mathfrak{p} \in \mathfrak{C}_m$ (in fact, dh$(M_{\mathfrak{p}})=0$), since if $\mathfrak{p} \in \mathfrak{C}_m$, one obtains by the characterization ($b_m$) of $m$-torsionless $R$-modules in the introduction that codh$(M_{\mathfrak{p}})=$ codh$(R_{\mathfrak{p}})$. Hence $M_{\mathfrak{p}}$ is free for every $\mathfrak{p} \in \mathfrak{C}_m$. As $M$ has rank $r$, we have that $\mu(M_{\mathfrak{p}})=\mu(M_{\mathfrak{q}})=r$ for all $\mathfrak{p}, \mathfrak{q} \in \mathfrak{C}_m$. The elements $g(e_1), \dots g(e_n)$ generate $M$, and so by Theorem 1, (2), there exist $a_1, \dots , a_{n-1} \in R$ such that if $x := e_n + \Sigma_{i=1}^{n-1}{a_ie_i}$, then $g(x)$ is not in $\mathfrak{p}M_{\mathfrak{p}}$, at all $\mathfrak{p} \in \mathfrak{C}_{m}$.\\

It is clear that for this choice of $x$, condition (1') is satisfied. Since $g(x)$ can be included as part of a basis of the free $R_{\mf{p}}$-module $M_{\mf{p}}$ for any $\mf{p} \in \mf{C}_{m}$ (and in particular, for any $\mf{p} \in $Ass $R)$, we see that $M''_{\mf{p}}$ is a free $R_{\mf{p}}$-module for these primes $\mf{p}$ and that $g(x)$ is a linearly independent element of $M$. Thus (2') follows.\\

If $\mf{p} \notin \mf{C}_m$ (i.e., codh$(R_{\mf{p}}) > m$), it follows from the exact sequence\\ 

\begin{center}
\mbox{%

\xymatrix{
{0} \ar[r] & {R_{\mf{p}}g(x)} \ar[r] & {M_{\mf{p}}} \ar[r] & {M''_{\mf{p}}} \ar[r] & {0}\\
}
}
\end{center}

\noindent that codh$(R_{\mf{p}}g(x))=$ codh$(R_{\mf{p}})>m$, codh$(M_{\mf{p}}) \geq m$ and codh$(M''_{\mf{p}}) \geq m$. Altogether, this gives codh$(M''_{\mf{p}}) \geq $ min$(m,$ codh$(R_{\mf{p}}))$ for all primes $\mf{p}$ in $R$. Since either $M''$ has finite homological dimension or $R$ is $m$-Gorenstein, it follows that $M''$ is $m$-torsionless. From (2') it follows that rank$(M'')=r-1$.\\

It only remains to verify (4'). It is elementary to show that $p'$ maps the kernel of $g$ onto the kernel of $g''$. But rank(Ker$(g))=$ rank$($Ker$(g''))$, whence Ker$(p'\mid_{Ker(g)})$ is a torsion-free $R$-module of rank $0$, and thus equal to the zero module.\\

\end{proof}

\begin{remark}
We may relax the hypotheses of Theorem 2 while ensuring that Theorem 1 is still applicable to get the following statement:\\

Suppose $M$ is an $R$-module such that $M_{\mf{p}}$ is free for $\mf{p} \in \mf{C}_m$ and such that $\mu(M_{\mf{p}}) \geq \mu(M_{\mf{q}})$ whenever codh$(R_{\mf{p}}) >$ codh$(R_{\mf{q}})$ for $\mf{p}, \mf{q} \in \mf{C}_{m}$. Further, let 

\begin{center}
$r:=$ min$\{\mu(M_{\mf{q}}) \mid \mf{q} \in $Ass $R \} \geq m,$\\ 
\end{center}

\noindent and $g, e_1, \dots , e_n, s$ be as in Theorem 2. Then there exist elements $x_s, \dots, x_n \in R^n$ satisfying conditions (1), (2) and (4) of Theorem 2 and such that the following condition (in place of (3)) holds: $(M/(\Sigma_{i=s}^{n}{Rg(x_i)}))_{\mf{p}}$ is free for $\mf{p} \in \mf{C}_m$ and there is $\mf{q} \in $Ass $R$ with $$\langle M/(\Sigma_{i=s}^{n}{Rg(x_i)}) \rangle _{\mf{q}} =m.$$

\noindent (The proof is similar to that of Theorem 2 and proceeds by induction on $r$, in which the conclusion of Theorem 1 serves as the inductive step. One requires only a slightly more subtle argument for (4): an $R$-homomorphism $f: N \rightarrow N'$, where $N$ is torsionfree, is injective if and only if for all $\mf{q} \in $Ass$ R$, $f \otimes_R R_{\mf{q}}$ is injective.)\\

\end{remark} 

The following corollary shows that under the slightly weaker hypotheses of Theorem 2 on $M$ (respectively $R$), the projective modules in an exact sequence:\\
 
\begin{center}
\mbox{%

\xymatrix{
{0} \ar[r] & {M} \ar[r] & {P_m} \ar[r] & {\dots} \ar[r] & {P_1}\\
}
}
\end{center}

\noindent can be chosen in the following prescribed manner:\\

\begin{corr}
Let $M$ be an $m$-torsionless $R$-module of well defined rank $r$, where $m \geq 2$. Let either \textup{dh}$(M)< \infty$ or $R_{\mf{p}}$ be a regular local ring for all $\mf{p} \in \mf{C}_{m-1}$. Then there exists an exact sequence\\

\begin{center}
\mbox{%

\xymatrix{
{0} \ar[r] & {M} \ar[r]^{f} & {R^{r+m-1}} \ar[r] & {R^{2m-3}} \ar[r] & {R^{2m-5}} \ar[r] & {\dots} \ar[r] & {R^3} \ar[r] & {R^1} .\\
}
}
\end{center}

\noindent In particular, $M$ admits an embedding $f: M \rightarrow R^{r+m-1}$ such that \textup{Coker}$(f)$ is $(m-1)$-torsionless and $M$ is the $(m-1)$-th syzygy module of an ideal of $R$. When $m>2$ this ideal is generated by three elements, and when $m=2$ it is generated by $r+1$ elements.\\

\end{corr}

\begin{proof}

Projective modules in the above exact sequence that satisfy property $(s_m)$ and are $m$-torsionless can without loss of generality be assumed to be free, so that there exists an exact sequence \\

\begin{center}
\mbox{

\xymatrix{
{0} \ar[r] & {M} \ar[r] & {F} \ar[r] & {N} \ar[r] & {0}\\
}
}
\end{center}

\noindent where $F$ is free and $N$ is $(m-1)$-torsionless. $N$ has a well defined rank and is of finite homological dimension if $M$ is. Applying Theorem 2 to the epimorphism $F \rightarrow N$ gives in the case that rank$(F) \geq r+m-1$ an exact sequence $$0 \rightarrow M \rightarrow F' \rightarrow N' \rightarrow 0,$$ where $N'$ is $(m-1)$-torsionless and of rank $m-1$, whence $F'$ is isomorphic to $R^{r+m-1}$. If  rank$(F) < r+m-1$ one obtains this exact sequence by adding to $F$ a free direct summand of the appropriate rank.

\end{proof}

The observation just made about $M$ yields for $N'$ an exact sequence 

$$0 \rightarrow N' \rightarrow R^{2m-3} \rightarrow N'' \rightarrow 0,$$ 

\noindent such that $N''$ is $(m-2)$-torsionless, etc.\\

A theorem of Bourbaki [2, pg. 76, Theor\`{e}me 6] is obtained as a special case of Theorem 2. Bourbaki proves the special case $m=1$ in the following corollary (under the additional assumption that $R$ does not contain any zero-divisors).\\

\begin{corr}

Let the localizations $R_{\mf{p}}$ be regular local rings for $\mf{p} \in \mf{C}_m$ and let $M$ be an $m$-torsionless $R$-module having a well-defined rank $\geq m$. Then there exists a free submodule $F$ of $M$, such that $M/F$ is $m$-torsionless and of rank $m$. 

\end{corr} 

\section{On the extension of certain projective resolutions to resolutions of ideals generated by three elements}

We are now prepared to give an easy proof of the following result that was highlighted in the introduction:\\

\begin{thm}
Let $M$ be an $m$-torsionless $R$-module having a well-defined rank, and let
 
\begin{center}
\mbox{

\xymatrix{
{0} \ar[r] & {F_n} \ar[r]^{f_n} & {F_{n-1}} \ar[r] & {\dots} \ar[r] & {F_{m+1}} \ar[r]^{f_{m+1}} & {F_{m}} \ar[r]^{g} & {M} \ar[r] & {0} \\
}
}
\end{center}

\noindent be a projective resolution of $M$, with $n>m \geq 1$, $F_m$ a free $R$-module and $r:= $\textup{rank(Im}$(f_{m+1}))$. Then there exist homomorphisms $c: F_m \rightarrow R^{r+m}, f_m: R^{r+m} \rightarrow R^{2m-1}, f_j: R^{2j+1} \rightarrow R^{2j-1}, j=1, \dots , m-1,$ such that if $f'_{m+1}:=c \circ f_{m+1}$, the sequence:\\

\begin{center}
\mbox{%

\xymatrix{
{0} \ar[r] & {F_n} \ar[r]^-{f_n} & {F_{n-1}} \ar[r] & {\dots} \ar[r] & {F_{m+2}} \ar[r]^-{f_{m+2}} & {F_{m+1}} \ar[r]^-{f'_{m+1}} & {R^{r+m}} \\ & \ar[r]^-{f_m} & {R^{2m-1}} & \ar[r]^{f_{m-1}} & R^{2m-3} & \ar[r] & {\dots} \ar[r] & {R^3} \ar[r]^-{f_1} & {R} \\
}
}
\end{center}

\noindent is exact. 
\end{thm}

\begin{proof}
If $u:= rank(F_m) \leq r+m$, we may choose $c$ to be the embedding $F_m \rightarrow F_m \oplus R^{r+m-u}$. Then\\

\begin{center}
\mbox{%

\xymatrix{
{0} \ar[r] & {F_n} \ar[r]^-{f_n} & {F_{n-1}} \ar[r] & {\dots} \ar[r] & {F_{m+2}} \ar[r]^-{f_{m+2}} & {F_{m+1}} \ar[r]^-{f'_{m+1}} & {R^{r+m}} \\
}
}
\end{center}

\noindent is exact, Coker$(f'_{m+1})$ is $m$-torsionless and of rank $m$.\\

If $u>r+m$, then by Theorem 2 there exist basis elements $x_s, \dots, x_u, s:= r+m+1$ of $F_m$ such that $M':= M/(Rg(x_s)+ \dots + Rg(x_u))$ is $m$-torsionless and of rank $m$. Further, the natural epimorphism $c: F_m \rightarrow F_m/(Rx_s + Rx_u) \cong R^{r+m}$ induces an exact sequence: $0 \rightarrow$ Im$(f_{m+1}) \cong$ Coker$(f_{m+2}) \rightarrow R^{r+m} \rightarrow M' \rightarrow 0$. Thus $f'_{m+1}$ has the required propery in this case as well. Finally, Corollary 1 to Theorem 2 gives the homomorphisms $f_i, i=1, \dots, m$.\\

\end{proof}

\begin{remark}
\begin{enumerate}
\item When the localizations $R_{\mf{p}}$ are regular local rings, Theorem 3 applies to infinite resolutions also. (See Theorem 2 and Corollary 1.)\\
\item It would suffice to prove Theorem 3 in the case that the modules $F_{n-1}, \dots, F_{m}$ are free and $F_n$ is projective and of well-defined rank, since $M$ has such a resolution. Conversely, if $M$ has such a resolution, then $M$ automatically has a well-defined rank.\\
\item If all the $F_i$ are free modules, then $M$ is $m$-torsionless if and only if grade$(I(f_j)) \geq j$ for all $j= m+1, \dots, n$. (The following definition of $I(f_j)$ is found in [3]: $I(f_j)$ is the (rank(Coker$(f_j)))$-th Fitting ideal of Coker$(f_j)$). This follows immediately from the characterization of $m$-torsionless $R$-modules of finite homological dimension ($b_m$) in the introduction.\\
\item In Theorem 2 (and its corollaries) $m$ is always a lower bound for the rank of the constructed $m$-torsionless $R$-modules. We conjecture that this is not a consequence of our method of proof, but is in fact because there are no non-projective $m$-torsionless $R$-modules of finite homological dimension and of rank $< m$. This is always true if either $m \leq 2$ or dh$(M) \leq 1$. To see this it is enough to consider the case when $R$ is local. The case $m=1$ is trivial. A $1$-torsionless, or even more generally, a $2$-torsionless $R$-module of rank $1$ is isomorphic to an ideal $\mf{a}$ of $R$, which under the given hypotheses has a finite free resolution and has grad$(\mf{a}) \geq 1$. But then $\mf{a}$ is isomorphic to an ideal $\mf{a'}$ with grade$(\mf{a'}) \geq 2$ ([8, Corollary 5.6] or [3, Corollary 5.2]). The ideal $\mf{a'}$ is not $2$-torsionless, unless $\mf{a'} = R$. When dh$(M) \leq 1$, our claim follows from known estimates on the degree of determinantal ideals: the degree of a (rank$(M))$-th Fitting ideal of $M$ is, for $M$ free, less than or equal to rank$(M)+1$ (see also [10, Theorem 2]).\\
\item {[3, Theorem 8.1]} derives a particular case of Theorem 3, where $n=3$, $m=2$, rank$(F_2)=$ rank$(F_1)+2$. The proof given there uses the structure theory of free resolutions developed in [3].\\
\end{enumerate}
\end{remark}

The following corollary says that every natural number that is the homological dimension of an $m$-torsionless $R$-module is also the homological dimension of a ($2m+1$)-generated $m$-torsionless $R$-module of rank $m$. The special case $m=1$ was proved using entirely different methods by Burch [4] and Kohn [7].\\

\begin{cor}
Let $s, m \geq 1$ be natural numbers such that there exists an $R$-module $N$ with \textup{dh}$(N)=s+m$. Then there is an $m$-torsionless $R$-module $M$ with \textup{dh}$(M)=s$, having rank $m$, and generated by at most $2m+1$ elements.
\end{cor}

\begin{proof}
Let $n:=s+m$. Since dh$(N)=n$, there exists an $R$-regular sequence $x_1, \dots, x_n$. When $s=1$, we let $M=R^n/R(x_1, \dots x_n)$. If $s>1$, consider the following section of the Koszul complex associated to $(x_1, \dots, x_n)$:\\

\begin{center}
\mbox{%

\xymatrix{
{0} \ar[r] & {R} \ar[r]^-{f_n} & {R^n} \ar[r]^-{f_{n-1}} & {\wedge^{2}R^n} \ar[r] & {\dots} \ar[r] & {\wedge^{s-2}R^n} \ar[r]^-{f_{m+2}} & {\wedge^{s-1}R^n}.\\
}
}
\end{center}

\noindent Coker$(f_{m+2})$ is $(m+1)$-torsionless and has a well-defined rank; let $r:=$ rank(Im$(f_{m+2}))$. By Theorem 3, there are homomorphisms $f'_{m+2}, f_{m+1}$ such that the complex

\begin{center}
\mbox{%

\xymatrix{
{0} \ar[r] & {R} \ar[r]^-{f_n} & {R^n} \ar[r] & {\dots} \ar[r] & {\wedge^{s-3}R^n} \ar[r]^-{f_{m+3}} & {\wedge^{s-2}R^n}\\ \ar[r]^-{f'_{m+2}} & {R^{r+m+1}} \ar[r]^-{f_{m+1}} & {R^{2m+1}}\\
}
}
\end{center}

\noindent is exact and Coker$(f_{m+1})$ is $m$-torsionless. It is clear that the claims pertaining to the	rank and the number of generators of $M$ hold. When $s=2$, we have $f'_{m+2}=f_n$ (indeed, rank(Coker$(f_n))=m+1$) and when $s>2$, $f_n$ is independent of the construction of $f'_{m+2}$. Since $f_n$ does not split, we have dh$(M)=s$.\\

If the conjecture made in Remark 4 to Theorem 3 holds, then there is an $m$-torsionless $R$-module of finite homological dimension (in fact of homological dimension $\leq 1$), generated by less than $2m+1$ elements. \footnote{ \emph{Addendum \& Correction}. D. Eisenbud has informed us that the conjecture stated in Remark 4 to Theorem 3 was first made by P. Hackman in a currently unpublished article ``Exterior Powers and Homology''. The referee requested that we make a reference to the article ``Tout ideal premier d'un anneau noeth\'{e}rien est associ\'{e} \`{a} un ideal engendr\'{e} par trois \'{e}l\'{e}ments'' by T. Gulliksen (\emph{C. R. Acad. Sci. Paris Ser. A} \textbf{271} (1970), 1207-1208). The main result of the paper stated in its title follows immediately from the corollary to Theorem 3.} (It is sufficient, again, to consider only local rings $R$; then $M$ has a well-defined rank whenever dh$(M) < \infty$.)

\end{proof}

\end{document}